\documentclass[a4paper]{amsart}
\usepackage{amsmath,amsthm,amssymb,amsfonts}
\usepackage{proof}
\usepackage[british]{babel}
\usepackage{pdfpages}
\usepackage[pdfborder={0 0 0}]{hyperref}
\usepackage{enumerate}
\usepackage{mathrsfs}
\usepackage{xypic}
\usepackage{stmaryrd}
\usepackage{bussproofs}
\usepackage{microtype}

\newtheorem{thm}{Theorem}[section]
\newtheorem*{thm*}{Theorem}

\newtheorem{lem}[thm]{Lemma}
\newtheorem{prop}[thm]{Proposition}

\theoremstyle{definition}
\newtheorem{defi}[thm]{Definition}
\newtheorem{rem}[thm]{Remark}

\newcommand{\seqto}{\vdash}
\newcommand{\lgodel}{\ulcorner}
\newcommand{\rgodel}{\urcorner}
\newcommand{\realises}{\mathbin{\mathrm{real}}}
\newcommand{\realisesp}{\mathbin{\mathrm{real}^+}}
\newcommand{\realisesm}{\mathbin{\mathrm{real}^-}}
\newcommand{\realisess}{\mathbin{\mathrm{real}^s}}
\newcommand{\realisest}{\mathbin{\mathrm{real}^t}}

\newcommand\Item[1][]{%
  \ifx\relax#1\relax  \item \else \item[#1] \fi
  \abovedisplayskip=0pt\abovedisplayshortskip=0pt~\vspace*{-\baselineskip}}
  
\title{On Weihrauch reducibility and intuitionistic reverse mathematics}
\author[R. Kuyper]{Rutger Kuyper}
\address{Department of Mathematics\\
University of Wisconsin--Madison\\
Madison, WI 53706\\
USA}
\email{mail@rutgerkuyper.com}

\begin{document}

\begin{abstract}
We show that there is a strong connection between Weihrauch reducibility on one hand, and provability in $\mathrm{EL}_0$, the intuitionistic version of $\mathrm{RCA}_0$, on the other hand. More precisely, we show that Weihrauch reducibility to the composition of finitely many instances of a theorem is captured by provability in $\mathrm{EL}_0$ together with Markov's principle, and that Weihrauch reducibility is captured by an affine subsystem of $\mathrm{EL}_0$ plus Markov's principle.
\end{abstract}

\maketitle

\section{Introduction}

There are two main approaches to classifying the relative strength of theorems in mathematics, where we are usually interested in statements which can be formulated as $\Pi^1_2$-formulas, i.e.\ formulas of the form $\forall X \xi(X) \to \exists Y \psi(X,Y)$. One of them, in the form of \emph{reverse mathematics}, classifies theorems by looking at their relative strength over a weak system, $\mathrm{RCA}_0$. For more background on reverse mathematics we refer to Simpson \cite{simpson-1999}.

The other approach, in the form of \emph{Weihrauch reducibility}, classifies theorems by their uniform computational power. In this case, we say that $\zeta_0$ is Weihrauch reducible to $\zeta_1$ if we can uniformly, computably transform instances of $\zeta_0$ into instances of $\zeta_1$, and given any solution for this instance of $\zeta_1$ we can uniformly compute a solution to $\zeta_0$ for the original instance from $\zeta_1$ and this original instance. For more background on Weihrauch reducibility we refer to Brattka and Gherardi \cite{brattka-gherardi-2011}.

Until recently, research in these two fields has been mostly disjoint. Dorais, Dzhafarov, Hirst, Mileti and Shafer \cite{ddhms-2015} and Dzhafarov \cite{dzhafarov-2015} try to bring together these two fields, by formalising Weihrauch reducibility within $\mathrm{RCA}_0$. This shows that Weihrauch-reducibility within $\mathrm{RCA}_0$ is a special case of reverse mathematics, namely the case where one has to prove $\zeta_0$ uniformly using only one application of $\zeta_1$.

On the other hand, several people have studied the connections between provability of sequential versions of $\Pi^1_2$-formulas in $\mathrm{RCA}_0$, and provability in $\mathrm{EL}_0$. This research was started by Hirst and Mummert \cite{hirst-mummert-2011}, and continued by Dorais \cite{dorais-2014} and Fujiwara \cite{fujiwara-2015}.  Their proofs use techniques from realisability, a field which has long studied the connections between intuitionistic logic and computability.

In the current paper, we combine these two research directions to show that there is a tight connection between Weihrauch reducibility and provability in $\mathrm{EL}_0$. Recall that, in the Weihrauch degrees, there is a natural notion of composition, denoted by $\star$, which was introduced by Brattka, Gherardi and Marcone \cite{bgm-2012} and shown to exist by Brattka, Oliva and Pauly \cite{bop-2013}. Here, $\zeta_0$ is Weihrauch-reducible to $\zeta_2 \star \zeta_1$ if and only if we can solve $\zeta_0$ by transforming an instance $X$ of $\zeta_0$ into an instance $Y$ of $\zeta_1$, and then given a solution for $Y$, transforming this into an instance $Y'$ for $\zeta_2$, and then given any solution for $Y'$, transforming this into a solution for $X$, all in a computable and uniform way. Note that there are three computable transformations taking place, and hence such a reduction is witnessed by three indices $e_1,e_2$ and $e_3$.

The main theorem of this paper is that there is some computable transformation $\zeta \mapsto \zeta'$ of $\Pi^1_2$-formulas satifying that $\zeta$ is classically equivalent to $\zeta'$, for which the following are equivalent:
\begin{enumerate}
\item There are an $n \in \omega$ and $e_1,\dots,e_{n+1}$ such that $\mathrm{RCA}_0$ proves that $e_1,\dots,e_{n+1}$ are indices for Turing functionals witnessing that $\zeta_0$ Weihrauch-reduces to the composition of $n$ copies of $\zeta_1$.
\item $\mathrm{EL}_0 + \mathrm{MP}$ proves that $\zeta'_1 \to \zeta'_0$.
\end{enumerate}
Here, recall that $\mathrm{MP}$ is \emph{Markov's principle}, that is, $\neg\neg \exists x \phi \to \exists x \phi$ for atomic formulas $\phi$. 

While this result shows that intuitionistic provability is very close to Weihrauch reducibility, in the sense that it captures the uniformity expressed by Weihrauch reducibility, it does not capture the other aspect which makes Weihrauch reducibility special, namely that it only allows one to use one instance of $\zeta_1$. However, in this paper we also prove that, if we suitably restrict our sequent calculus to an affine version $\mathrm{IQC}^{\exists\alpha\mathrm{a}}$, to be defined below, we get the following result:
\begin{enumerate}
\item There are $e_1,e_2$ such that $\mathrm{RCA}_0$ proves that $e_1,e_2$ witness that $\zeta_0$ Weihrauch-reduces to $\zeta_1$.
\item $(\mathrm{EL}_0 + \mathrm{MP})^{\exists\alpha\mathrm{a}}$ proves that $\zeta'_1 \to \zeta'_0$.
\end{enumerate}

This shows that Weihrauch-reducibility can be seen as a proof-theoretic notion. That is, proving that $\zeta_0$ is Weihrauch-reducible to $\zeta_1$ in $\mathrm{RCA}_0$ can be seen as proving that $\zeta'_0$ implies $\zeta'_1$ using a restricted set of derivation rules.

The paper is structured as follows. In the next section, we will introduce the sequent calculus for intuitionistic predicate logic, or $\mathrm{IQC}$, that we will be using throughout this paper, and we will introduce $\mathrm{EL}_0$. Next, in section \ref{sec-weihrauch} we will briefly talk about Weihrauch reducibility and how we formalise this notion within $\mathrm{EL}_0$. After that, in section \ref{sec-real} we will introduce a specially tailored reducibility notion, which will allow us to extract Weihrauch reductions from proofs in $\mathrm{EL}_0$. Then, in section \ref{sec-to-weihrauch} we show how this extraction works, and therefore prove the first implication of the first main result mentioned above. The converse of this is proven in section \ref{sec-converse}. Finally, in the last section we show how to capture Weihrauch reducibility using $\mathrm{IQC}^{\exists\alpha\mathrm{a}}$.

Our notation is mostly standard. We let $\omega$ denote the (standard) natural numbers. For undefined notions from proof theory, we refer to Buss \cite{buss-1998}.

\section{$\mathrm{IQC}$ and $\mathrm{EL}_0$}

In this section, we will formalise the sequent calculus for $\mathrm{IQC}$ that we will be working with in this paper, since the proofs depend on the exact calculus we take. After that, we will discuss $\mathrm{EL}_0$ and the affine version mentioned in the introduction.

\begin{defi}\label{defi-iqc}
We say that a sequent $\Gamma \seqto \phi$ is derivable in $\mathrm{IQC}$ if it is derivable using the following rules.

\begin{prooftree}
\RightLabel{$(I)$}
\AxiomC{}
\UnaryInfC{$A \seqto A$}
\end{prooftree}
\begin{minipage}[t]{0.5\textwidth}
\begin{prooftree}
\RightLabel{$(\wedge L)$}
\AxiomC{$\Gamma,\psi_1,\psi_2 \seqto \phi$}
\UnaryInfC{$\Gamma,\psi_1 \wedge \psi_2 \seqto \phi$}
\end{prooftree}

\begin{prooftree}
\RightLabel{$({\to} L)$}
\AxiomC{$\Gamma_1 \seqto \psi_1$}
\AxiomC{$\Gamma_2,\psi_2 \seqto \phi$}
\BinaryInfC{$\Gamma_1,\Gamma_2,\psi_1 \to \psi_2 \seqto \phi$}
\end{prooftree}

\begin{prooftree}
\RightLabel{$(\forall L)$}
\AxiomC{$\Gamma,\psi[x := t] \seqto \phi$}
\UnaryInfC{$\Gamma,\forall x \psi \seqto \phi$}
\end{prooftree}
\end{minipage}
\begin{minipage}[t]{0.5\textwidth}
\begin{prooftree}
\RightLabel{$(\wedge R)$}
\AxiomC{$\Gamma_1 \seqto \phi_1$}
\AxiomC{$\Gamma_2 \seqto \phi_2$}
\BinaryInfC{$\Gamma_1,\Gamma_2 \seqto \phi_1 \wedge \phi_2$}
\end{prooftree}

\begin{prooftree}
\RightLabel{$({\to} R)$}
\AxiomC{$\Gamma,\phi_1 \seqto \phi_2$}
\UnaryInfC{$\Gamma \seqto \phi_1 \to \phi_2$}
\end{prooftree}

\begin{prooftree}
\RightLabel{$(\forall R)$}
\AxiomC{$\Gamma \seqto \phi[x := y]$}
\UnaryInfC{$\Gamma \seqto \forall x \phi$}
\end{prooftree}
\end{minipage}

\begin{minipage}[t]{0.5\textwidth}
\begin{prooftree}
\RightLabel{$(\exists L)$}
\AxiomC{$\Gamma,\psi[x := y] \seqto \phi$}
\UnaryInfC{$\Gamma,\exists x \psi \seqto \phi$}
\end{prooftree}

\begin{prooftree}
\RightLabel{$(W)$}
\AxiomC{$\Gamma \seqto \phi$}
\UnaryInfC{$\Gamma,\psi \seqto \phi$}
\end{prooftree}
\begin{prooftree}
\RightLabel{$(C)$}
\AxiomC{$\Gamma,\psi,\psi \seqto \phi$}
\UnaryInfC{$\Gamma,\psi \seqto \phi$}
\end{prooftree}

\begin{prooftree}
\RightLabel{$(P)$}
\AxiomC{$\Gamma_1,\psi_1,\psi_2,\Gamma_2 \seqto \phi$}
\UnaryInfC{$\Gamma_1,\psi_2,\psi_1,\Gamma_2 \seqto \phi$}
\end{prooftree}
\end{minipage}
\begin{minipage}[t]{0.5\textwidth}
\begin{prooftree}
\RightLabel{$(\exists R)$}
\AxiomC{$\Gamma \seqto \phi[x := t]$}
\UnaryInfC{$\Gamma \seqto \exists x \phi$}
\end{prooftree}

\begin{prooftree}
\RightLabel{$(\bot)$}
\AxiomC{$\Gamma \seqto \bot$}
\UnaryInfC{$\Gamma \seqto A$}
\end{prooftree}
\end{minipage}

\begin{prooftree}
\RightLabel{(Cut)}
\AxiomC{$\Gamma_1 \seqto \psi$}
\AxiomC{$\Gamma_2,\psi \seqto \phi$}
\BinaryInfC{$\Gamma_1,\Gamma_2 \seqto \phi$}
\end{prooftree}

Here, $A$ is an atomic formula, $y$ does not appear free in $\Gamma$ and $\phi$ in $(\exists L)$, and $y$ does not appear free in $\Gamma$ in $\forall R$. Furthermore, in case $\psi_2$ in $({\to} L)$ is $\bot$, we only allow this rule if $\phi$ is also $\bot$ (we will also call this rule $(\neg L))$.
\end{defi}

As commonly done in intuitionistic logic, we do not have any rules for $\neg$, because $\neg$ is not a primitive connective: $\neg\phi$ is interpreted as $\phi \to \bot$.
However, note that our calculus also does not contain any rules for $\vee$. This is because, as is usual in $\mathrm{EL}$ and $\mathrm{EL}_0$, we do not see $\vee$ as a primitive connective either, but instead define $\phi \vee \psi$ as $\exists x ((x = 0 \to \phi) \wedge (x \not= 0 \to \psi))$. This is justified by the fact that equality is decidable in $\mathrm{EL}_0$, see Troelstra \cite{troelstra-1973}.
Furthermore, in $\mathrm{EL}$ this definition for $\vee$ is equivalent to the usual intuitionistic disjunction, see \cite[section 1.3.7]{troelstra-1973}. 

However, it is important to note that this definition of $\vee$ is \emph{weaker} than the usual intuitionistic disjunction, because the proof of their equivalence uses arithmetical induction, which we do not have available in $\mathrm{EL}_0$. We will come back to this in Remark \ref{rem-disjunction} below, after we have formalised $\mathrm{EL}_0$.

Let us adopt the following standard convention: \emph{every variable is bound at most once, and no variable can be both bound and free simultaneously}. Furthermore, let us adopt the convention that once a bound variable disappears from a proof because of an application of the cut rule, then we are not allowed to reuse that variable later in the proof.

Now, let us define $\mathrm{EL}_0$, following Dorais \cite{dorais-2014} (who mostly follows \cite{troelstra-1973}). Unlike the usual formulation of $\mathrm{RCA}_0$, this is done using functions instead of sets. Let us denote number variables by $x,y,\dots$ and function variables by $\alpha,\beta,\dots$. The language contains the constant $0$, the successor function $S$, equality, and a function symbol for every primitive recursive function. There are two kinds of terms: number terms $t$ and function terms $\tau$. They are formed as follows:
\begin{itemize}
\item Every number variable is a number term.
\item Every function variable is a function term.
\item The zero constant $0$ is a number term.
\item If $t$ is a number term, then so is $S(t)$.
\item If $t_1,\dots,t_n$ are number terms and $f$ is a symbol for an $n$-ary primitive recursive function, then $f(t_1,\dots,t_n)$ is a number term.
\item If $t$ is a number term and $\tau$ is a function term, then the evaluation $\tau(t)$ is a number term.
\item If $t$ is a number term and $x$ is a number variable, then the lambda abstraction $\lambda x.t$ is a function term.
\item If $t$ is a number term and $\gamma$ is a function term, then the recursion $\mathrm{R}t \tau$ is a function term.
\end{itemize}

We only have one atomic relation in our language: equality on the number sort. Equality on the function sort is defined by extensionality, i.e.\
\[\tau_1 = \tau_2 \leftrightarrow \forall x(\tau_1(x) = \tau_2(x)).\]

\begin{defi}\label{defi-el0}
Let $\Gamma \seqto \phi$ be a sequent in the language described above. Then we say that $\Gamma \seqto \phi$ is derivable in $\mathrm{EL}_0$ if it is derivable in $\mathrm{IQC}$ using the equality axioms, the defining axioms for all the primitive recursive function symbols, and the following axioms.

\begin{equation}
S(x) \not= 0 \wedge (S(x) = S(y) \to x = y) \tag{SA}
\end{equation}

\begin{equation}
\forall x(B(x) \to B(S(x))) \to \forall y(B(0) \to B(y)) \tag{QF-IA}
\end{equation}

\begin{equation}
(\lambda x.t)(t') = t[x := t'] \tag{CON}
\end{equation}

\begin{equation}
(\mathrm{R}t \tau)(0) = t \wedge (\mathrm{R}t \tau)(S(t')) = \tau((\mathrm{R}t \tau)(t')) \tag{REC}
\end{equation}

\begin{equation}
\forall x \exists y B(x,y) \to \exists \alpha \forall z B(z,\alpha(z)) \tag{QF-AC$^{0,0}$}
\end{equation}

Here, $B$ is a quantifier-free formula.
\end{defi}

Since our system has symbols for all primitive recursive functions, we can encode finite sequences of numbers in the usual manner. We will implicitly do this throughout this paper.

Although $\mathrm{RCA}_0$ is usually formulated using sets instead of functions, there is a natural connection between them, since we can identify a function with its graph, and a set with its indicator function. Indeed, if we add the \emph{law of the excluded middle}, i.e.\ $\phi \vee \neg\phi$, to $\mathrm{EL}_0$, we get a system that is equivalent to $\mathrm{RCA}_0$ under the identification just described (see Dorais \cite{dorais-2014} or Kohlenbach \cite{kohlenbach-2005}). A particular ingredient of this proof is that $\mathrm{EL}_0$ in fact allows induction (QF-IA) and choice (QF-AC$^{0,0}$) for not just quantifier-free formulas $B$, but even for $\Sigma^0_1$-formulas $\phi$.

\begin{prop}
Let $\phi(x) = \exists y B(x,y)$, where $B$ is quantifier-free. Then $\mathrm{EL}_0$ proves
\[\forall x(\phi(x) \to \phi(S(x))) \to \forall x(\phi(0) \to \phi(x))\]
and
\[\forall x \exists y \phi(x,y) \to \exists \alpha \forall x \phi(x,\alpha(x)).\]
\end{prop}
\begin{proof}
Consider the relation $C(x,y)$ which holds if and only if both $y$ is a sequence of length $x+1$, and for every $z \leq x$ we have that $B(x,y(z))$ holds. Note that this relation is primitive recursive in $B$, so using the recursion operator we can find a function term $\tau$ representing the indicator function of this relation.\footnote{Note that $B$ might contain function variables, so we cannot directly use the fact that the language contains a symbol for every primitive recursive function.} Let us reason within $\mathrm{EL}_0$. Assume $\forall x(\phi(x) \to \phi(S(x)))$ and $\phi(0)$ both hold. Then, using QF-IA, we know that $\forall x \exists y( \tau(x,y)=1)$ holds. Thus, by QF-AC$^{0,0}$ there is a function $f$ such that $\forall x (\tau(x,f(x))=1)$ holds. In particular, we have that $\forall x (B(x,f(x)(x)))$ holds, and therefore $\forall x (\phi(x))$ holds.

That $\mathrm{EL}_0$ proves the axiom of choice for $\Sigma^0_1$-formulas follows directly using a pairing function.
\end{proof}

\begin{rem}\label{rem-disjunction}
As noted above, we do not see $\vee$ as a primitive connective, but instead interpret $\phi \vee \psi$ as  $\exists x ((x = 0 \to \phi) \wedge (x \not= 0 \to \psi))$. We also noted that this interpretation is equivalent to the usual intuitionistic disjunction in $\mathrm{EL}$, i.e.\ $\mathrm{EL}_0$ with full induction. However, in $\mathrm{EL}_0$ this is not true, because we do not have enough induction to prove this. In $\mathrm{EL}_0$ only a weak version of the usual $\vee L$ rule holds: we generally only have that
\begin{prooftree}
\AxiomC{$\Gamma,\psi_1 \seqto \phi$}
\AxiomC{$\Gamma,\psi_2 \seqto \phi$}
\BinaryInfC{$\Gamma,\psi_1 \vee \psi_2 \seqto \phi$}
\end{prooftree}
holds for $\Sigma^0_1$-formulas $\phi$, and for formulas $\phi$ starting with a negation. In fact, our proofs implicitly exploit this fact. Fortunately, all the other usual logical laws, such as the commutativity of $\vee$, still hold.
\end{rem}

There is one additional principle which, although not part of $\mathrm{EL}_0$, will be used in this paper.

\begin{defi}
\emph{Markov's principle}, or $\mathrm{MP}$, is the principle
\[\neg\neg \exists x B(x) \to \exists y B(y)\]
for quantifier-free formulas $B$.
\end{defi}

As for $\mathrm{RCA}_0$, much of elementary computability theory can be carried out inside $\mathrm{EL}_0$; see Kleene \cite{kleene-1969}\footnote{Although Kleene uses a different system, as mentioned in Troelstra \cite[p.\ 73]{troelstra-1973}, his developments apply to our system as well.} and Troelstra \cite{troelstra-1973} for many details.

In particular, the normal form theorem holds, in the following form. Here, $\alpha_1 \oplus \dots \oplus \alpha_n$ denotes the function which sends $kn+i$ to $\alpha_{i-1}(k)$, and just like for the pairing functions for numbers, $\mathrm{EL}_0$ proves all relevant properties about this join. Below, $\lgodel e \rgodel$ denotes the term representing the number $e$.

\begin{thm}{\rm (\cite{kleene-1969})}\label{thm-normal-form}
There exists a quantifier-free formula $\phi(y_1,y_2,\beta,y_3,y_4,y_5)$ and a term $s(z)$ such that for every function term $\tau(\alpha_1,\dots,\alpha_n,x_1,\dots,x_m)$ there exists an $e \in \omega$ such that in $\mathrm{EL}_0$ we have:
\begin{align*}
&\forall \alpha_1, \dots, \alpha_n \forall x_1, \dots ,x_m \forall a,b\\
&\;\;\tau(\alpha_1,\dots,\alpha_n,x_1,\dots,x_m)(a) = b\\
&\quad\leftrightarrow \exists k (\phi(\lgodel e \rgodel,\langle n,m\rangle,\alpha_1 \oplus \dots \oplus \alpha_n,\langle x_1,\dots,x_m\rangle,a,k) \wedge s(k) = b)\\
&\quad\leftrightarrow \forall k (\phi(\lgodel e \rgodel,\langle n,m\rangle,\alpha_1 \oplus \dots \oplus \alpha_n,\langle x_1,\dots,x_m\rangle,a,k) \to s(k) = b).
\end{align*}
\end{thm}
\begin{proof}
This follows from Lemma 41 and the discussion on p.\ 69 of \cite{kleene-1969}.
\end{proof}

In view of this proposition, we will write
\[\Phi_e(\alpha_1,\dots,\alpha_n,x_1,\dots,x_m)(a) = b\]
for the formula
\[\exists k (\phi(\lgodel e \rgodel,\langle n,m\rangle,\alpha_1 \oplus \dots \oplus \alpha_n,\langle x_1,\dots,x_m\rangle,a,k) \wedge s(k) = b).\]
Also, we will write
\[\Phi_e(\alpha_1,\dots,\alpha_n,x_1,\dots,x_m) = \beta\]
for the formula
\[\forall a (\Phi_e(\alpha_1,\dots,\alpha_n,x_1,\dots,x_m)(a) = \beta(a)).\]
Finally, we will write
\[\Phi_e(\alpha_1,\dots,\alpha_n,x_1,\dots,x_m) = c\]
for the formula 
\[\Phi_e(\alpha_1,\dots,\alpha_n,x_1,\dots,x_m)(0) = c.\]

For some of our results below, we need to look at a weaker, affine version of $\mathrm{IQC}$. Here, by affine we mean that we restrict the use of the contraction rule $(C)$. Alternatively, we could have formalised this system by a suitable embedding into linear logic, but we think the current approach is more intuitive.

\begin{defi}
We say that a sequent $\Gamma \seqto \phi$ is derivable in $\mathrm{IQC}^\mathrm{a}$ if it is derivable using the rules from Definition \ref{defi-iqc}, except for $(C)$.
\end{defi}

Finally, let us define two affine versions of $\mathrm{EL}_0$.

\begin{defi}
We define $\mathrm{EL}_0^\mathrm{a}$ as in Definition \ref{defi-el0}, except using $\mathrm{IQC}^\mathrm{a}$ instead of $\mathrm{IQC}$. 
Furthermore, we say that a sequent $\Gamma \seqto \phi$ is derivable in $\mathrm{EL}_0^{\exists\alpha\mathrm{a}}$ if it is derivable from the axioms in Definition \ref{defi-el0} using the rules from Definition \ref{defi-iqc},  except we only allow $(C)$ for formulas not containing function quantifiers, and we only allow $(W)$ for subformulas of formulas of the form $\exists \alpha \psi$ where $\psi$ does not contain any function quantifiers. We define $(\mathrm{EL}_0 + \mathrm{MP})^\mathrm{a}$ and $(\mathrm{EL}_0 + \mathrm{MP})^\mathrm{\alpha a}$ in a similar way.
\end{defi}

\section{Weihrauch reducibility}\label{sec-weihrauch}

In this section we will discuss how Weihrauch reducibility can be formalised within $\mathrm{EL}_0$. Weihrauch reducibility is normally defined using represented spaces, see e.g. Brattka and Gherardi \cite{brattka-gherardi-2011}. However, we are specifically looking at problems given by formulas of the form $\forall \alpha \xi(\alpha) \to \exists \beta \psi(\alpha,\beta)$ (which we will henceforth call $\Pi^1_2$-formulas), in which case Weihrauch reducibility can be defined in an easier way (see also Dorais et al. \cite{ddhms-2015} and Dzhafarov \cite{dzhafarov-2015}). So, let us define Weihrauch reducibility within $\mathrm{EL}_0$ as follows.

\begin{defi}
Let $\zeta_0 = \forall \alpha_0 \xi_0(\alpha_0) \to \exists \beta_0 \psi_0(\alpha_0,\beta_0)$ and $\zeta_1 = \forall \alpha_1 \xi_1(\alpha_0) \to \exists \beta_1 \psi_1(\alpha_1,\beta_1)$ be $\Pi^1_2$-formulas. Then we define, in $\mathrm{EL}_0$, that $\zeta_0$ \emph{Weihrauch-reduces to} $\zeta_1$ if there exists $e_1,e_2$ such that
\begin{align*}
&\forall \alpha_0(\xi_0(\alpha_0) \to \Phi_{e_1}(\alpha_0){\downarrow} \wedge \xi_1(\Phi_{e_1}(\alpha_0)))\\
\wedge &\forall \alpha_0 \forall \beta_0 ((\xi_0(\alpha_0) \wedge \psi_1(\Phi_{e_1}(\alpha_0),\beta_0))\\
&\quad\to \Phi_{e_2}(\alpha_0,\Phi_{e_1}(\alpha_0),\beta_0){\downarrow} \wedge \psi_0(\alpha_0,\Phi_{e_2}(\alpha_0,\Phi_{e_1}(\alpha_0),\beta_0))).
\end{align*}
\end{defi}

As discussed in the introduction, there is a notion of composition, the \emph{compositional product}, on Weihrauch degrees. We would like to define the composition of two $\Pi^1_2$-formulas in a way which corresponds to this notion in the Weihrauch degrees, but it does not seem like this can be done using a $\Pi^1_2$-formula. Instead, we define what it means to be reducible to a composition.

\begin{defi}
Let $\zeta_0 = \forall \alpha_0 \xi_0(\alpha_0) \to \exists \beta_0 \psi_0(\alpha_0,\beta_0), \dots, \zeta_n = \forall \alpha_n \xi_n(\alpha_n) \to \exists \beta_n \psi_n(\alpha_n,\beta_n)$ be $\Pi^1_2$-formulas. Then we define, in $\mathrm{EL}_0$, that $\zeta_0$ \emph{Weihrauch-reduces to the composition of} $\zeta_n,\dots,\zeta_1$ if there exist $e_1,\dots,e_{n+1}$ such that, for all $\alpha_0$ and for all $\beta_1,\dots,\beta_n$, and for every $1 \leq i \leq n$, we have that, if for $1 \leq j \leq n$ we write $\alpha_j$ for $\Phi_{e_{j}}(\alpha_0,\alpha_1,\dots,\alpha_{j-1},\beta_1,\dots,\beta_{j-1})$, then, if for all $1 < j \leq i$ we have that
$\psi_j(\alpha_j,\beta_j)$ holds, and $\xi_0(\alpha_0)$ holds, then
\[\Phi_{e_i}(\alpha_0,\alpha_1,\dots,\alpha_{i-1},\beta_1,\dots,\beta_{i-1}) {\downarrow}\]
holds and $\xi_i(\alpha_i)$ holds;
and finally, that if for all $1 \leq j \leq n$ we have that $\psi_j(\alpha_j,\beta_j)$ holds, then
we have that
\[\Phi_{e_{n+1}}(\alpha_0,\alpha_1,\dots,\alpha_{n},\beta_1,\dots,\beta_{n}) {\downarrow}\]
holds, and that
\[\psi_0(\alpha_0,\Phi_{e_{n+1}}(\alpha_0,\alpha_1,\dots,\alpha_{n},\beta_1,\dots,\beta_{n}))\]
holds.

We say that $e_1,\dots,e_{n+1}$ \emph{witness} the Weihrauch reduction.
\end{defi}

It is not hard to verify that $\zeta_0$ Weihrauch-reduces to the composition of $\zeta_1$ if and only if $\zeta_0$ Weihrauch-reduces to $\zeta_1$, as expected.

If we see each $\zeta_n,\dots,\zeta_1$ as a black box transforming instances into solutions, Figure \ref{figure-1} illustrates what it means for $\zeta_0$ to Weihrauch-reduce to the composition of $\zeta_n,\dots,\zeta_1$.\footnote{Note that $\Phi_{e_i}$ is allowed to use all of the functions appearing earlier in the sequence. For example, $\Phi_{e_2}$ is allowed to use $\alpha_0,\alpha_1$ and $\beta_1$, and not just $\beta_1$, although the diagram might suggest otherwise.}

\begin{figure}
\label{figure-1}
\xymatrixcolsep{3pc}
\centerline{
\xymatrix{
\alpha_0 \ar[r]^{\Phi_{e_1}} & \alpha_1 \ar[d]_{\zeta_1} & \alpha_2 \ar[d]_{\zeta_2} & \dots & & \alpha_n \ar[d]_{\zeta_n} \\
\beta_0 & \beta_1 \ar[ur]^{\Phi_{e_2}} & \beta_2 & \dots &\beta_{n-1} \ar[ur]^{\Phi_{e_n}} & \beta_n \ar@/^2pc/[lllll]^{\Phi_{e_{n+1}}}}}
\caption{A Weihrauch reduction of $\zeta_0$ to the composition of $\zeta_n,\dots,\zeta_1$.}
\end{figure}
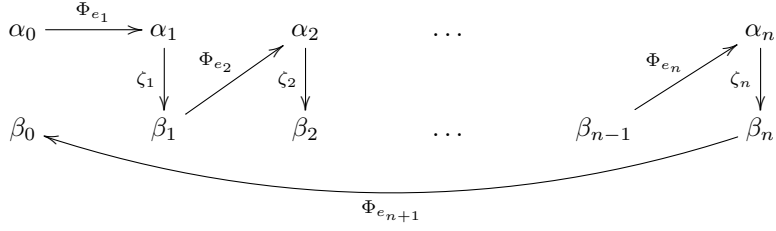

%Finally, let us say what it means to reduce to $!\zeta$, whose notation is motivated by the \emph{bang} operation in linear logic. The idea is that something reduces to $!\zeta$ if and only if it reduces to the composition of finitely many copies of $\zeta$. Here, `finitely many' includes the possibility that it reduces to the composition of zero copies of $\zeta$, which we define to be the identity function (i.e.\ the formula $\forall \alpha \exists \beta \alpha=\beta$).
%It is not hard to see that, for any two $\Pi^1_2$-formulas $\zeta_1,\zeta_2$, there is a single formula $\phi(n)$ which expresses that $\zeta_1$ is reducible to the composition of $n$ many copies of $\zeta_2$. Therefore, the following definition makes sense.

%\begin{defi}
%Let $\zeta_1,\zeta_2$ be $\Pi^1_2$-formulas. Then we define, in $\mathrm{EL}_0$, that $\zeta_1$ reduces to $!\zeta_2$ if $\zeta_1$ reduces to the composition of finitely many copies of $\zeta_2$.
%\end{defi}

\section{Realisability}\label{sec-real}

To prove our result, we will introduce a specially tailored notion of realisability. Our realisers will be pairs of functions of the form $v,w: \{1,\dots,n\} \to \omega \cup (\omega \times \omega)$. Let us explain the motivation behind them. First, let us consider $w$. The idea is that, if $w(i) \in \omega \times \omega$, say $w(i) = (k,e)$, then $w$ says that the intended value for $x_i$ can be computed using the Turing functional $\Phi_e$. Similarly, if $v(i) = (k,e)$, then $v$ says that the intended value for $\alpha_i$ can be computed using the Turing functional $\Phi_e$.
Furthermore, we want to carefully bookkeep which other variables it is allowed to use in the computation, which is what we use $k$ for. Let us define $w_1(j)$ as the first coordinate of $w(j)$ if $w(j) \in \omega \times \omega$, and as $w(j)$ otherwise. Similarly we can define $v_1$.

Then, in the computation of $x_i$, $\Phi_e$ is only allowed to use those $\alpha_i$ with $v(\alpha_i) < k$ and those $x_j$ with $w_1(j) < k$. Similarly, if $v(i) = (k,e)$, then we can compute the intended value for $\alpha_i$ using $\alpha_j$ with $v_1(j) < k$, and $x_j$ with $w_1(j) < k$.

Let us formalise this idea.

\begin{defi}
Let $v: \{1,\dots,n\} \to \omega \cup (\omega \times \omega)$ and $w: \{1,\dots,n'\} \to \omega \cup (\omega \times \omega)$. We say that a sequence of
functions $f_1,\dots,f_n$ and a sequence of numbers $a_1,\dots,a_{n'}$ are \emph{valid for $(v,w)$ at $v(m)$} if, whenever $v(m) \in \omega \times \omega$, say $v(m) = (k,e)$, if $i_1,\dots,i_s$ lists in increasing order those $1 \leq i \leq n$ with $v_1(i) < k$, and $j_1,\dots,j_t$ lists in increasing order those $1 \leq j \leq n'$ with $v(j) < k$, then we have that, if $\Phi_e(f_{i_1},\dots,f_{i_s},a_{i_1},\dots,a_{i_t}) {\downarrow}$, then it is equal to $f_m$. We say that it is \emph{strongly valid for $(v,w)$ at $v(m)$} if it is both valid for $(v,w)$ at $v(m)$, and $\Phi_e(f_{i_1},\dots,f_{i_s},a_{i_1},\dots,a_{i_t}) {\downarrow}$.

Similarly, we define the notions of valid and strongly valid for $(v,w)$ at $w(m)$ using $a_m$.

Now, we say that $(f_1,\dots,f_n,a_1,\dots,a_n)$ is \emph{valid for $(v,w)$} if it is valid at $v(m)$ for all $1 \leq m \leq n$ and it is valid for $w(m)$ at all $1 \leq m \leq n'$.

Finally, if whenever $v(m) = (k,e)$ or $w(m) = (k,e)$, if for all $i$ with $v(i) = (k',e')$ we have that the computation of the functional $\Phi_{e}$ is independent of $f_i$ (i.e.\ it ignores the input $f_i$), and for all $j$ with $w(j) = (k',e')$ we have that the computation is independent of $a_j$, then we will say that $(v,w)$ is \emph{self-contained}.
\end{defi}

Note that we can identify
\[\bigcup_{n,n' \in \omega} (\{1,\dots,n\} \to \omega \cup (\omega \times \omega)) \times (\{1,\dots,n'\} \to \omega \cup (\omega \times \omega))\]
with the set of natural numbers in a primitive recursive way, and hence in a way which is definable within $\mathrm{EL}_0$. We will tacitly assume this. Given any $v$, we let $|v|$ be the number $n$ such that $v$ is defined on $\{1,\dots,n\}$, and similarly for $w$.

In our proofs, we will be considering $\Pi^1_2$-formulas. Let us therefore make the following definition.

\begin{defi}
A $\Pi^1_2$-sequent is a sequent $\Gamma \seqto \phi$ such that all formulas in $\Gamma$ and $\phi$ are $\Pi^1_2$, i.e.\ of the form $\forall \alpha \xi(\alpha) \to \exists \beta \psi(\alpha,\beta)$ where $\psi$ and $\xi$ are arithmetical.
\end{defi}

To make the proof of our main theorem work, we will need to make the contractions used in the proof explicit. For this, we use the following definition, together with two lemmas.

\begin{defi}
Let $\phi,\psi$ be formulas. We say that $\phi$ is \emph{contraction-similar to $\psi$} if there exists a sequence $\psi = \phi_0,\phi_1,\dots,\phi_n = \phi$ such that for every $1 \leq i \leq n$, $\phi_{i}$ is obtained from $\phi_{i-1}$ by replacing a subformula $\chi$ of $\phi_{i-1}$ by $\chi \wedge \chi$.

Furthermore, we say that a finite set $\Sigma$ of formulas is \emph{contraction-similar} to a finite set $\Gamma$ of formulas if there exists a surjection $f$ of $\Sigma$ onto $\Gamma$ such that every formula $\phi$ in $\Sigma$ is contraction-similar to $f(\phi)$. We will say that $\phi$ and $f(\phi)$ are \emph{linked}.

Finally, a sequent $\Sigma \seqto \phi$ is \emph{contraction-similar to $\Gamma \seqto \psi$} if $\Sigma$ is contraction-similar to $\Gamma$ and $\phi$ is contraction-similar to $\psi$.
\end{defi}

The proof of the first lemma is straightforward.

\begin{lem}\label{lem-con-sim-equiv}
If $\phi$ is contraction-similar to $\psi$, then $\phi \leftrightarrow \psi$ is provable in $\mathrm{IQC}$.
\end{lem}

On the other hand, as long as we replace our sequents by suitable contraction-similar sequents, we can transfer proofs from $\mathrm{EL}_0$ to $\mathrm{EL}_0^\mathrm{a}$. Below we will, however, also need that the proofs can be sufficiently cut free. Recall that a proof is \emph{free-cut free} if the principal formulas of any cut appearing in the proof are axioms, e.g.\ a proof in $\mathrm{EL}_0$ is free-cut free if every principal formula of a cut appearing in the proof is one of the axioms from Definition \ref{defi-el0}.

\begin{lem}\label{lem-con-sim}
Let $\Sigma$ be the set of formulas which are contraction-similar to an axiom of $\mathrm{EL}_0$.
If a $\Pi^1_2$-sequent $\Gamma \seqto \phi$ is provable in $\mathrm{EL}_0$, then there is a free-cut free proof in $\mathrm{EL}_0^\mathrm{a}$ from $\Sigma$ of some $\Pi^1_2$-sequent $\Gamma' \seqto \phi'$ which is contraction-similar to $\Gamma \seqto \phi$. The same holds if we replace $\mathrm{EL}_0$ by $\mathrm{EL}_0 + \mathrm{MP}$ and $\mathrm{EL}_0^\mathrm{a}$ by $(\mathrm{EL}_0 + \mathrm{MP})^\mathrm{a}$.
\end{lem}
\begin{proof}
We prove this by induction on a free-cut free proof of $\Gamma \seqto \phi$.

The crucial step is when the last step in the proof is an instance of the contraction rule.
Thus, we have a proof of $\Gamma,\psi,\psi \seqto \phi$. Let us first assume that $\psi$ does not contain an existential function quantifier. By the induction hypothesis, we therefore have a proof of some contraction-similar sequent $\Gamma',\zeta,\xi \seqto \phi'$ where $\zeta$ and $\xi$ are contraction-similar to $\psi$. But then we also have a proof of $\Gamma',\zeta \wedge \xi \seqto \phi'$ (renaming some bound variables, if necessary), and this sequent is contraction-similar to $\Gamma,\psi \seqto \phi$.

If $\psi$ does contain an existential function quantifier, then $\Gamma',\zeta,\xi \seqto \phi'$ is already contraction-similar to $\Gamma,\psi \seqto \phi$, if we declare $\zeta$ and $\xi$ to be linked.

Let us also consider the case when the last step in the proof is an instance of the $\forall \alpha_i L$ rule. In this case, list the formulas $\zeta_1,\dots,\zeta_n$ which are linked to the principal formula $\psi$, and apply an instance of the $\forall \alpha_i L$ rule to each $\zeta_i$.

Next, consider the case when the last step in the proof is an instance of ${\to} L$, and the principal formula $\psi$ contains a function quantifier. In this case, the proof is similar to the one given in the previous paragraph, so we omit the details.

Finally, when we encounter a cut on an axiom of $\mathrm{EL}_0$, we replace this by a cut on a formula from $\Sigma$.
\end{proof}
 
We are now going to define our realisability notion.
That is, we are going to define a binary formula
\[\lgodel (v,w) \rgodel \realises \lgodel \Gamma \seqto \phi\rgodel,\]
which we will also write as
\[(v,w) \realises \Gamma \seqto \phi.\]
To do so, we first need to define what it means to be realised at some sequence $(f_1,\dots,f_{|v|},a_1,\dots,a_{|w|})$, of which we will need a positive and negative version.
We do this by recursion on the formulas (which we can do in $\mathrm{EL}_0$, using the recursion operator).

\begin{defi}
We simultaneously define
\[(v,w),(f_1,\dots,f_{|v|},a_1,\dots,a_{|w|}) \realisesp \phi\]
and
\[(v,w),(f_1,\dots,f_{|v|},a_1,\dots,a_{|w|}) \realisesm \phi\]
recursively as follows (where $s \in \{+,-\}$ and $Q \in \{\exists,\forall\}$):
\begin{enumerate}
\item If $\phi$ is atomic, then
\[(v,w),(f_1,\dots,f_{|v|},a_1,\dots,a_{|w|}) \realisess \phi\]
is the conjunction of the formula which expresses that (the universal closure of) $\phi$ is true, that $\phi$ has both free and bound second-order variables contained in $\{\alpha_1,\dots,\alpha_{|v|}\}$ and first-order variables contained in $\{x_1,\dots,x_{|w|}\}$, and of the formula which expresses that $\phi[\alpha_i := f_i, x_j := a_j]$ is true.

\Item
\[(v,w),(f_1,\dots,f_{|v|},a_1,\dots,a_{|w|}) \realisess \phi \wedge \psi\]
if (the universal closure of) $\phi \wedge \psi$ is true,
\[(v,w),(f_1,\dots,f_{|v|},a_1,\dots,a_{|w|}) \realisess \phi\]
and
\[(v,w),(f_1,\dots,f_{|v|},a_1,\dots,a_{|w|}) \realisess \psi.\]

\Item
\[(v,w),(f_1,\dots,f_{|v|},a_1,\dots,a_{|w|}) \realisess \phi \to \psi\]
if, $\phi \to \psi$ is true, and whenever
\[(v,w),(f_1,\dots,f_{|v|},a_1,\dots,a_{|w|}) \realisest \phi\]
we have 
\[(v,w),(f_1,\dots,f_{|v|},a_1,\dots,a_{|w|}) \realisess \psi,\]
where $t$ is the opposing sign of $s$.

\Item
\[(v,w),(f_1,\dots,f_{|v|},a_1,\dots,a_{|w|}) \realisesp \exists x_i \phi(x_i)\]
if
$\exists x_i \phi(x_i)$ is true,
\[(v,w),(f_1,\dots,f_{|v|},a_1,\dots,a_{|w|}) \realisesp \phi(x_i),\]
and $(f_1,\dots,f_{|v|},a_1,\dots,a_{|w|})$ is strongly valid for $(v,w)$ at $w(i)$.

\Item
\[(v,w),(f_1,\dots,f_{|v|},a_1,\dots,a_{|w|}) \realisesp \forall x_i \phi(x_i)\]
if $\forall x_i \phi(x_i)$ is true and
\[(v,w),(f_1,\dots,f_{|v|},a_1,\dots,a_{|w|}) \realisesp \phi(x_i).\]

\Item
\[(v,w),(f_1,\dots,f_{|v|},a_1,\dots,a_{|w|}) \realisesm Q x_i \phi(x_i)\]
if
$Q x_i \phi(x_i)$ is true and
\[(v,w),(f_1,\dots,f_{|v|},a_1,\dots,a_{|w|}) \realisesm \phi(x_i).\]

\Item
\[(v,w),(f_1,\dots,f_{|v|},a_1,\dots,a_{|w|}) \realisesp \exists \alpha_i \phi(\alpha_i)\]
if $\exists \alpha_i \phi(\alpha_i)$ is true,
\[(v,w),(f_1,\dots,f_{|v|},a_1,\dots,a_{|w|}) \realisesp \phi(\alpha_i),\]
and $(f_1,\dots,f_{|v|},a_1,\dots,a_{|w|})$ is strongly valid for $(v,w)$ at $v(i)$,

\Item
\[(v,w),(f_1,\dots,f_{|v|},a_1,\dots,a_{|w|}) \realisesp \forall \alpha_i \phi(\alpha_i)\]
if $\forall \alpha_i \phi(\alpha_i)$ is true and
\[(v,w),(f_1,\dots,f_{|v|},a_1,\dots,a_{|w|}) \realisesp \phi(\alpha_i).\]

\Item
\[(v,w),(f_1,\dots,f_{|v|},a_1,\dots,a_{|w|}) \realisesm Q \alpha_i \phi(\alpha_i)\]
if $Q \alpha_i \phi(\alpha_i)$ is true and
\[(v,w),(f_1,\dots,f_{|v|},a_1,\dots,a_{|w|}) \realisesm \phi(\alpha_i).\]
\end{enumerate}
\end{defi}

\begin{defi}
Given a sequent $\Gamma \seqto \phi$ we say that $(v,w)$ is \emph{monotone} for $\Gamma \seqto \phi$ if:
\begin{itemize}
\item For every subformula of $\Gamma$ or $\phi$ of the form $Q x_i \psi$, if $x_j$ is bound in $\psi$, then $w_1(i) < w_1(j)$, and if $\alpha_j$ is bound in $\psi$, then $w_1(i) < v_1(j)$.
\item For every subformula of $\Gamma$ or $\phi$ of the form $Q \alpha_i \psi$, if $x_j$ is bound in $\psi$, then $v_1(i) < w_1(j)$, and if $\alpha_j$ is bound in $\psi$, then $v_1(i) < v_1(j)$.
\end{itemize}

Furthermore, we say that $(v,w)$ is \emph{universal} for $\Gamma \seqto \phi$ if for all $x_i$ bound by a universal quantifier in a positive position (see Buss \cite[1.2.10]{buss-1998}) or by an existential quantifier in a negative position we have that $w(i) \in \omega$, and similarly for such $\alpha_i$ that $v(i) \in \omega$.

We now let
\[(v,w) \realises \Gamma \seqto \phi\]
be the formula expressing that $v$ and $w$ are monotone and universal for $\Gamma \seqto \phi$, and that for all sequences $(f_1,\dots,f_{|v|},a_1,\dots,a_{|w|})$ which are valid for $(v,w)$ we have that, if
\[(v,w),(f_1,\dots,f_{|v|},a_1,\dots,a_{|w|}) \realisesm \psi\]
holds for all $\psi \in \Gamma$, then
\[(v,w),(f_1,\dots,f_{|v|},a_1,\dots,a_{|w|}) \realisesp \phi\]
also holds.

Finally, if $\Gamma \seqto \phi$ is a $\Pi^1_2$-sequent, say
\[\Gamma = \forall \alpha_3 \xi_1(\alpha_3) \to \exists \alpha_4 \psi_1(\alpha_3,\alpha_4),\dots,\forall \alpha_{2n+1} \xi_{n}(\alpha_{2n+1}) \to \exists \alpha_{2n+2} \psi_{1}(\alpha_{2n+1},\alpha_{2n+2})\]
and $\phi = \forall \alpha_1 \xi_0(\alpha_1) \to \exists \alpha_2 \psi_0(\alpha_1,\alpha_2)$, then we say that \emph{$(v,w)$ strongly realises} $\Gamma \seqto \phi$ if, when we let $m$ be maximal such that $v_1(2m)  \leq v_1(2)$, then
\begin{itemize}
\item $v_1(3) < v_1(5) < \dots < v_1(2m-1)$;
\item $(v,w) \realises \Gamma \seqto \phi$;
\item $v_1(i) < w_1(j)$ whenever $w_1(j) > 0$ and $1 \leq i \leq m$;
\item for all sequences $(f_1,\dots,f_{|v|},a_1,\dots,a_{|w|})$ which are valid for $(v,w)$ and for all $1 \leq i \leq m$ we have that, if
$\xi_0(f_1)$ is true, and
$\psi_j(f_{2j+1},f_{2j+2})$ is true for every $1 \leq j < i$, then
$\xi_i(f_{2i+1})$ is true, and $(f_1,\dots,f_{|v|},a_1,\dots,a_{|w|})$ is strongly valid for $(v,w)$ at $v(2i+1)$.
\item for all sequences $(f_1,\dots,f_{|v|},a_1,\dots,a_{|w|})$ which are valid for $(v,w)$, if
$\xi_0(f_1)$ is true, and
$\psi_j(f_{2j+1},f_{2j+2})$ is true for every $1 \leq j \leq n$, then
$(f_1,\dots,f_{|v|},a_1,\dots,a_{|w|})$ is strongly valid for $(v,w)$ at $v(2)$.
\end{itemize}
\end{defi}

Now that we have given the definition of our realisability notion, we will show that sequents which are provable in $\mathrm{EL}_0^\mathrm{a}$ have a realiser, and even provably so.

\begin{thm}\label{thm-realiser}
Let $\Sigma$ be the set of formulas which are contraction-similar to an axiom of $\mathrm{EL}_0$.
Let $\Gamma \seqto \phi$ be a sequent provable in $\mathrm{EL}_0^\mathrm{a}$ from $\Sigma$ by a free-cut free proof. Then there is a realiser $(v,w)$ of $\Gamma \seqto \phi$, and in fact $\mathrm{RCA}_0$ proves that $(v,w) \realises \Gamma \seqto \phi$.

The same holds if we replace $\mathrm{EL}_0$ by $\mathrm{EL}_0 + \mathrm{MP}$ and $\mathrm{EL}_0^\mathrm{a}$ by $(\mathrm{EL}_0 + \mathrm{MP})^\mathrm{a}$
\end{thm}
\begin{proof}
We prove this using induction on free-cut free proofs. In fact, we prove something slightly stronger: we even show that we can take $v(i)=0$ for non-bound variables $\alpha_i$ and $v(j) = 0$ for non-bound variables $x_j$. Furthermore, we prove that we can assume that if some variable $\alpha_i$ is bound we have that $v_1(i) > 0$, and similarly if some variable $x_j$ is bound then $w_1(j) > 0$. Finally, we will prove that we can take $(v,w)$ to be self-contained.

\begin{enumerate}
\item Identity $(I)$ or equality axioms\footnote{We assume the equality axioms are formulated without quantifiers, which we can do because they are universal.}: we can take $v$ and $w$ to be constantly $0$ on a domain which is large enough to contain all variables.
\item Weakening: we can obtain a realiser $(v',w')$ for the lower sequent by extending the realiser $(v,w)$ for the upper sequent to a big enough domain, and by choosing suitable values for $v'(i)$ and $w'(j)$ for bound variables occurring in the principal formula so as not to violate the extra conditions imposed in the induction; and letting $v'(i) = v(i)$ and $w'(j) = w(j)$ for other $i \leq |v|$ and $j \leq |w|$, and finally letting $v'(i) = 0$ and $w'(j) = 0$ for $i > |v|, j > |w|$ which do not represent any bound variables in the principal formula. It is not hard to see that this can be done and we omit the details.
\item Permutation: the realiser for the upper sequent is also a realiser for the lower sequent.
\item All non-quantifier logical rules with one upper sequent: we can use the same realiser for the lower sequent as for the upper sequent.
\item \label{step-imp} All non-quantifier logical rules with two upper sequents: since no variable occurs twice, and since non-bound variables are mapped to $0$, we can first restrict the realiser $(v^1,w^1)$ that we get from the induction hypothesis for the left upper sequent to the variables occurring in the left upper sequent, and the realiser $(v^2,w^2)$ for the right upper sequent to the variables occurring in the right upper sequent, and then their domains only contain free variables on which they agree. Then as a realiser $(v',w')$ for the lower sequent we can just take the union $(v^1 \cup v^2,w^1 \cup w^2)$ of these realisers with a restricted domain, where we add zeroes to extend their domains to an initial segment of $\omega$.\footnote{Note that this does not work for the usual $\vee L$ rule for intuitionistic logic, since there is an implicit contraction happening in the application of this rule. However, our logic does not contain this rule because of our choice of interpreting $\phi \vee \psi$ as $\exists x (x = 0 \to \phi \wedge x \not= 0 \to \psi)$.}

Here, let us note that we have a lot of extra freedom in choosing our realiser $(v',w')$ of the lower sequent. For example, we could assume that all non-zero values of $v^1_1,w^1_1$ are strictly above the non-zero values of $v^2_1,w^2_1$ by scaling them before taking the union. We will use this fact below.

\item $\exists x_u R$: let $(v,w)$ realise the upper sequent. Now, let $i_1,\dots,i_s$ list those $1 \leq i \leq |v|$ with $v(i) = 0$ in increasing order, and let $j_1,\dots,j_t$ list those $1 \leq j \leq |w|$ with $w(j) = 0$ in increasing order.
Note that for every free variable $\alpha_i$ we have $v(i) = 0$, and that for every free variable $x_j$ we have $w(j) = 0$. So, let $t(\alpha_{i_1},\dots,\alpha_{i_s},x_{j_1},\dots,x_{j_t})$ be the term which is captured by $\exists x_j R$. By Theorem \ref{thm-normal-form} we can find an index $e$ for the Turing functional mapping $(f_1,\dots,f_s,a_1,\dots,a_t) \in (\omega^\omega)^s \times \omega^t$ to $t(f_1,\dots,f_s,a_1,\dots,a_t)$. Now, given a realiser $(v,w)$ for the upper sequent, we let $w'(u) = (1,e)$, we let $v_1'(i) = v_1(i) + 1$ if $v_1(i) \not= 0$ and we let $w_1'(j) = w_1(j) + 1$ if $w_1(j) \not= 0$ and $j \not= u$, and we let $v'$ and $w'$ agree with $v$ and $w$ otherwise.
\item $\forall x_u L$:  this step is analogous to $\exists x_u R$.
\item $\forall x_u R$: let $(v,w)$ realises the upper sequent. Now we let $w'(u) = 1$, we let $v_1'(i) = v_1(i) + 1$ if $v_1(i) \not= 0$ and we let $w_1'(j) = w_1(j) + 1$ if $w_1(j) \not= 0$ and $j \not= u$, and we let $v'$ and $w'$ agree with $v$ and $w$ otherwise. Then $(v',w')$ realises the lower sequent.
\item $\exists x_u L$: this step is analogous to $\forall x_u R$.

\item $\exists \alpha_u R$: this step is analogous to $\exists x_u R$.
\item $\forall \alpha_u L$: this step is analogous to $\exists \alpha_u R$.
\item $\forall \alpha_u R$: this step is analogous to $\forall x_u R$.
\item $\exists \alpha_u L$: this step is analogous to $\forall \alpha_u R$.
\item An axiom which is an element of $\Sigma$: let us consider an instance of QF-AC$^{0,0}$; the other formulas are similar (cf.\ the induction steps for cuts below). Thus, consider the formula
$\phi = \forall x_c \exists x_d \psi(x_c,x_d) \to \exists \alpha_i \forall x_b \psi(x_b,\alpha_i(x_b))$. 
Then we get a realiser $(v,w)$ of $\seqto \phi$ by setting $w(c) = 3$, $w(d) = 4$, $w(i) = (1,e)$ and $w(b) = 2$, where $e$ is an index for the computable function which, on input $x_b$, computes the least $m$ such that $\psi(x_b,m)$ holds, which we can find because $\psi$ is quantifier-free. Note that $e$ is an index for a total computable function if $\forall x_c \exists x_d \psi(x_c,x_d)$ is true, which shows that $(v,w)$ realises $\seqto \phi$.

\item Cuts on a cut formula which is not contraction-similar to an instance of $\mathrm{IA}_0$, QF-AC$^{0,0}$ or $\mathrm{MP}$: it can be directly verified that every realiser of the right upper sequent also realises the lower sequent, using the fact that all axioms which are not instances of $\mathrm{IA}_0$ or QF-AC$^{0,0}$ do not have any bound variables.

\item Cuts on a cut formula which is contraction-similar to an instance of IA$_0$: let us first consider the case where the cut formula is actually an instance of IA$_0$. Thus, let us say that the cut formula is of the form
\[\phi = \forall x_i (\psi(x_i) \to \psi(S(x_i))) \to \forall x_j(\psi(0) \to \psi(x_j)),\]
where $\psi$ is quantifier-free. Let $(v,w)$ realise the right upper sequent, let $w'(i) = w'(j) = 0$, and let $v'(s) = v(s)$ and $w'(s) = w(s)$ otherwise.

We claim that $(v',w')$ realises the lower sequent. In fact, we argue that $(v',w')$ negatively realises $\phi$, and after that we argue that this is enough to prove the claim.
Let us reason within EL$_0$.
Let $(f_1,\dots,f_{|v|},a_1,\dots,a_{|w|})$ be valid for $(v',w')$. Assume $(v',w')$ positively realises $\forall x_i(\psi(x_i) \to \psi(S(x_i)))$ at the sequence $(f_1,\dots,f_{|v|},a_1,\dots,a_{|w|})$. In particular, it is true. So, using induction we get that $\forall x_j(\psi(0) \to \psi(x_j))$ is true. In particular, $\psi_0 \to \psi(a_j)$ holds, and since $\psi$ is quantifier-free this means that $\forall x_j(\psi_0 \to \psi(x_j))$ is negatively realised by $(v',w')$ at $(f_1,\dots,f_{|v|},a_1,\dots,a_{|w|})$. So, $(v',w')$ negatively realises $\phi$.

Now, if $\sigma = (f_1,\dots,f_{|v|},a_1,\dots,a_{|w|})$ is valid for $(v',w')$, then we can turn this into a sequence $\sigma'$ which is valid for $(v,w)$ by replacing $a_j$ with the number computed from the other elements of $\sigma$ using $\Phi_{w_2(j)}$ if $w(j) \in \omega \times \omega$ and this computation converges (here we use that $(v,w)$ is self-contained). Furthermore, $\sigma'$ remains valid for $(v',w')$.
%Here we seem to use the distinction whether the machine converges or not, so we seem to need classical logic
Let $\Delta \seqto \eta$ be the conclusion of the application of the cut rule.
Then, if $(v',w')$ negatively realises $\Delta$ at $\sigma$, so does $(v,w)$ at $\sigma'$ (again, using that $(v,w)$ is self-contained), and we just argued that it negatively realises $\phi$ at $\sigma'$ as well, so by our choice of $(v,w)$ it positively realises $\eta$ at $\sigma'$. We therefore see that $(v',w')$ positively realises $\eta$ at $\sigma$, as desired.

Now, if instead we only have that the cut formula is contraction-similar to an instance of IA$_0$, this means that at least one subformula $\chi$ has been replaced by $\chi \wedge \chi$. Let us assume there is exactly one such subformula; the general case then follows by induction. If $\chi$ is quantifier-free, we can use exactly the same proof as we gave in the previous paragraphs. If $\chi$ is an instance of IA$_0$, we get the result by applying the transformation we just described two times, since we can see $\chi \wedge \chi$ as two instances of IA$_0$. If $\chi = \forall x_i(\psi(x_i) \to \psi(S(x_i)))$, i.e.\ the cut formula is of the form
\begin{align*}
\phi = &(\forall x_i (\psi(x_i) \to \psi(S(x_i))) \wedge \forall x_{i'} (\psi(x_{i'})\\
&\to \psi(S(x_{i'})))) \to \forall x_j(\psi(0) \to \psi(x_j)),
\end{align*}
the only changes we need to make is to make $w'(i') = 0$. The case where $\chi = \forall x_j(\psi(0) \to \psi(x_j))$ is similar. This covers all the cases.

\item Cuts on a cut formula which is contraction-similar an instance of QF-AC$^{0,0}$: again, let us first say that the cut formula is actually an instance of QF-AC$^{0,0}$, i.e.\ of the form
\[\phi = \forall x_c \exists x_d \psi(x_c,x_d) \to \exists \alpha_i \forall x_b \psi(x_b,\alpha_i(x_b)),\]
where $\psi$ is quantifier-free; the general case follows as in the previous case.
Without loss of generality, the last step in the proof of the right upper sequent had $\phi$ as its principal formula, since we can always move the cuts upwards in the proof tree as much as possible. Furthermore, if $\phi$ was introduced in the right upper sequent through a weakening (i.e.\ $\phi$ was \emph{only weakly introduced}, see \cite[2.4.4.2]{buss-1998}), we can remove the weakening and the cut rule from the proof.
Thus, we may assume we have a proof of the form
\begin{prooftree}
\AxiomC{$\phi$}
\AxiomC{$\Delta \seqto \forall x_c \exists x_d \psi(x_c,x_d)$}
\AxiomC{$\Delta', \exists \alpha_i \forall x_b \psi(x_b,\alpha_i(x_b)) \seqto \eta$}
\BinaryInfC{$\Delta,\Delta',\phi \seqto \eta$}
\BinaryInfC{$\Delta,\Delta' \seqto \eta$}
\end{prooftree}
Let $(v,w)$ be a realiser for $\Delta', \exists \alpha_i \forall x_b \psi(x_b,\alpha_i(x_b)) \seqto \eta$. Let $e$ be an index for the functional which, given valuations for the free variables in $\psi(x_c,x_d)$ different from $x_c$ and $x_d$, on input $x_c$, looks for the least value $m$ such that $\psi(x_c,m)$ holds. We can do this because $\psi$ is quantifier free. Now for every $j \not= i$, if $v(j)$ is of the form $(k,e')$, find a new index $e''$ which ignores the input $\alpha_i$ and replaces it by $\Phi_e$; and do the same thing for $w(j)$ for every $j \not= b$. Finally, let $v'(i) = w'(b) = 0$.

We claim: $(v',w')$ realises $\Delta,\Delta' \seqto \eta$. Let $(f_1,\dots,f_{|v'|},u_1,\dots,u_{|w'|})$ be valid for $(v',w')$. Assume $(v',w')$ negatively realises $\Delta$ and $\Delta'$ at this sequence.
Then, in particular, all formulas in $\Delta$ are true, so $\forall x_c \exists x_d \psi(x_c,x_d)$ is true and therefore $\Phi_e$ is total. Furthermore, since it negatively realises $\Delta'$, it is now not hard to argue that, by replacing $a_b$ in a similar way as in the previous case, we have that $(v,w)$ positively realises $\eta$, as desired.

\item Cuts on a cut formula which is contraction-similar to an instance of MP: let us say that the cut formula is of the form $\phi = \neg\neg \exists x_b \psi(x_b) \to \exists x_c \psi(x_c)$, where $\psi$ is quantifier-free; the general case follows in the same way as above.
Again, without loss of generality, the last step in the proof of the right upper sequent had $\phi$ as its principal formula; thus, we have a proof of the form
\begin{prooftree}
\AxiomC{$\seqto \phi$}
\AxiomC{$\Delta \seqto \neg\neg \exists x_b \psi(x_b)$}
\AxiomC{$\Delta', \exists x_c \psi(x_c) \seqto \eta$}
\BinaryInfC{$\Delta,\Delta',\phi \seqto \eta$}
\BinaryInfC{$\Delta,\Delta' \seqto \eta$}
\end{prooftree}

As before, let $(v,w)$ realise $\Delta', \exists x_c \psi(x_c) \seqto \eta$. Let $e$ be an index for the functional which computes the least number $m$ such that $\psi(m)$ holds, which we can do because $\psi$ is quantifier-free. Then, constructing $(v',w')$ using a similar replacement strategy as in the previous case, we get a realiser $(v',w')$ for $\Delta,\Delta' \seqto \psi$.\qedhere
\end{enumerate}
\end{proof}

Just having a realiser is, however, not enough to be able to extract a Weihrauch reduction. For this we need a strong realiser, and we will now prove these exist in a special case.

\begin{defi}
Let
\[\zeta_0 = \forall \alpha_0 \xi_0(\alpha_0) \to \exists \beta_0 \psi_0(\alpha_0,\beta_0), \dots, \zeta_n = \forall \alpha_n \xi_n(\alpha_n) \to \exists \beta_n \psi_n(\alpha_n,\beta_n)\]
be $\Pi^1_2$-formulas such that every first-order quantifier in $\xi_0$ and in $\psi_1,\dots,\psi_n$ is in the scope of a negation. Then we say that $\zeta_1,\dots,\zeta_n \seqto \zeta_0$ is a \emph{number-negative} $\Pi^1_2$-sequent.
\end{defi}

\begin{thm}\label{thm-strong-realiser}
Let $\Sigma$ be the set of formulas which are contraction-similar to an axiom of $\mathrm{EL}_0$.
Let $\Gamma \seqto \forall \alpha_1 \xi_0(\alpha_1) \to \exists \alpha_2 \psi_0(\alpha_1,\alpha_2)$ be a number-negative $\Pi^1_2$ sequent provable in $\mathrm{EL}_0^\mathrm{a}$ from $\Sigma$ with a free-cut free proof. 
Then there is a strong realiser $(v,w)$ of $\Gamma' \seqto \phi$ for some permutation $\Gamma'$ of a subsequence of $\Gamma$
(and in fact $\mathrm{RCA}_0$ proves that $(v,w) \realises \Gamma' \seqto \phi$).

The same holds if we replace $\mathrm{EL}_0^\mathrm{a}$ by $(\mathrm{EL}_0 + \mathrm{MP})^\mathrm{a}$.
\end{thm}
\begin{proof}
Let $\zeta_0 = \forall \alpha_1 \xi_0(\alpha_1) \to \exists \alpha_2 \psi_0(\alpha_1,\alpha_2)$.
Fix a free-cut free proof from $\Sigma$ of $\Gamma \seqto \zeta_0$ in $\mathrm{EL}_0^\mathrm{a}$. First, note that we may assume that all weakenings had as a principal formula a subformula of a formula of the form $\exists \alpha \eta$ where $\eta$ does not contain any function quantifiers, as long as we move to a subsequence $\Gamma'$ of $\Gamma$.

We now claim that the construction used in the proof of Theorem \ref{thm-realiser}, applied to this proof, directly gives us a strong realiser $(v,w)$. Let us assume that $\Gamma'$ consists of one formula $\zeta_1 = \forall \alpha_3 \xi_1(\alpha_3) \to \exists \alpha_4 \psi_1(\alpha_3,\alpha_4)$; the general case then follows by induction after we order the formulas in $\Gamma'$ according to the values of $v_1$ at their variables.

Let us assume that $v_1(1) < v_1(3) < v_1(4) < v_1(2)$; the other cases are similar but easier.
We need to show that $v_1(2) < w_1(i)$ for any $w_1(i) > 0$. There are four cases: $x_i$ occurs in $\psi_0,\psi_1,\xi_0$ or $\xi_1$. In the first three cases, we show that, if $x_i$ occurs in $\zeta_1$ or $\zeta_2$, then it is bound during the proof\footnote{When we talk about the stage at which a variable is bound during the proof, we mean the derivation step in the proof in whose conclusion the variable first appears bound.} before $\alpha_2$ is bound, which gives the result by closely inspecting the proof of Theorem \ref{thm-realiser}. In fact, it is not hard to see that, if $\alpha_i$ and $\alpha_j$ are bound on the same path of the proof tree, then $v_1(i) < v_1(j)$ if and only if $\alpha_i$ is bound before $\alpha_j$, and similarly for other combinations of function variables and number variables. The case where $x_i$ occurs in $\xi_1$ uses a slightly more complicated argument.

First, let us assume $x_i$ appears in $\psi_1$. Towards a contradiction, let us assume that $x_i$ is not yet bound in some sequent $s$ occurring in the proof, while $\alpha_2$ is bound. Then $x_i$ eventually has to become bound, and since $x_i$ is in the scope of a negation, there must be an application of $\neg L$ after the step in which this binding occurs and hence below $s$; say with conclusion $\Delta \seqto$. So $\exists \alpha_2 \psi_0(\alpha_1,\alpha_2)$ is a subformula of a formula in $\Delta$, and therefore we can never get to the conclusion $\Gamma \seqto \forall \alpha_1 \xi_0(\alpha_1) \to \exists \alpha_2 \phi_0(\alpha_1,\alpha_2)$ by the convention that every variable is bound at most once.

The case where $x_i$ appears in $\xi_0$ is similar. Also, if $x_i$ appears in $\psi_0$ then it is clearly bound before $\alpha_2$. Finally, let us assume $x_i$ appears in $\xi_1$. Then, by our assumption on the use of the weakening rule in our proof, we know that there is some application of ${\to} L$ of the form
\begin{prooftree}
\AxiomC{$\Delta \seqto \xi_1(\alpha_3)$}
\AxiomC{$\Delta',\exists \alpha_4 \psi_1(\alpha_3,\alpha_4) \seqto \eta$}
\BinaryInfC{$\Delta,\Delta',\xi_1 \to \exists \alpha_4 \psi_1(\alpha_3,\alpha_4) \seqto \eta$}
\end{prooftree}
As above, if $x_i$ is bound before $\alpha_2$, then the proof of Theorem \ref{thm-realiser} directly tells us that $v_1(2) < w_1(i)$. So, assume $x_i$ is not bound before $\alpha_2$. Note that $\alpha_2$ is bound in $\eta$, because $\alpha_4$ is bound and otherwise we would have $v_1(2) < v_1(4)$, contradicting our assumption. Then, because the quantifier binding $x_i$ is in the scope of a negation, we can argue as above that $x_i$ has to be bound in $\Delta,\Delta'$. Furthermore, if it is bound in $\Delta'$, it has to be bound before $\alpha_2$, contradicting our assumption. Thus, $x_i$ is bound in $\Delta$, and $\alpha_2$ is bound in $\eta$.
Then, as mentioned in step \eqref{step-imp} of the construction of the realiser, we may assume that $w_1(i) > v_1(2)$, as desired. Finally, using Lemma \ref{lem-induction} below we know that if $\xi_0$ is true, then $\Delta$ is true, and hence $\xi_1$ is true, which is one of the other things we needed to show to prove that $(v,w)$ is a strong realiser.

\bigskip
Next, let us assume that $\xi_0(f_1)$ and $\psi_1(f_3,f_4)$ are true. Let $v(3) = (k,e)$. We claim: $\Phi_e(f_1) {\downarrow}$. Indeed, the only way this could be false is if, somewhere during the proof, we used cut elimination on a formula which is contraction-similar to an instance of QF-AC$^{0,0}$, after $\alpha_3$ was bound. Let us consider an instance of QF-AC$^{0,0}$, say
\[\phi = \forall x_c \exists x_d \eta(x_c,x_d) \to \exists \alpha_i \forall x_b \eta(x_b,\alpha_i(x_b)),\]
the general case follows in the same way. Again, by moving the cut up if necessary, the proof directly before the cut is of the following form.
\begin{prooftree}
\AxiomC{$\seqto \phi$}
\AxiomC{$\Delta \seqto \forall x_c \exists x_d \psi(x_c,x_d)$}
\AxiomC{$\Delta', \exists \alpha_i \forall x_b \psi(x_b,\alpha_i(x_b)) \seqto \eta$}
\BinaryInfC{$\Delta,\Delta',\phi \seqto \eta$}
\BinaryInfC{$\Delta,\Delta' \seqto \eta$}
\end{prooftree}
Since we know that $\alpha_3$ was already bound at this stage of the proof, and we assumed that $v_1(3) < v_1(2)$, we know that $\alpha_2$ was also already bound at this stage of the proof. Therefore, again using Lemma \ref{lem-induction}, we have that $\Delta$ follows from $\xi_0(f_1)$ and the axioms from $\Sigma$, and hence from $\xi_0(f_1)$. Thus, $\forall x_c \exists x_d \psi(x_c,x_d)$ is true, and therefore $\Phi_e(f_1) {\downarrow}$.
By a similar argument, we see that $\Phi_e(f_1,f_3,f_4) {\downarrow}$, which completes the proof of the claim that $(v,w)$ is a strong realiser of the sequent $\zeta_1 \seqto \zeta_0$, as desired.

\bigskip
 
Finally, if we add MP, then nothing changes if we assume that, in our proof of $\Gamma \seqto \zeta_0$, whenever we use an axiom $\phi \in \Sigma$, it is directly followed by a cut rule with $\phi$ as its principal formula (if we actually want to use $\phi$ in the proof, we can prove $\phi \seqto \phi$ first and then apply the cut rule to get a proof of $\seqto \phi$ which does not violate this condition). Then we still have in the proof above that, after $x_i$ gets bound, there needs to be an application of $\neg L$ (even though there are axioms in $\Sigma$ which contain negation symbols).
\end{proof}

\begin{lem}\label{lem-induction}
If in $\mathrm{EL}_0 + \mathrm{MP}$ there is a proof of $\Delta,\Gamma \seqto \xi \to \exists \alpha \psi$ from $\Delta',\Gamma \seqto \exists \alpha \psi$, then, if $\Delta$, $\xi$ and all cut formulas are true,  so is $\Delta'$.
\end{lem}
\begin{proof}
By induction on a free-cut free proof. For the induction step, we look at the topmost derivation rule. Note that the only rule which can be applied to the right-hand side is ${\to} R$ with $\xi \to \exists \alpha \psi$ as its principal formula, for which the proof is clear. The proof is also straightforward when the topmost rule only has one upper sequent. Next, consider the case where we have an application of ${\to} L$ of the following form.
\begin{prooftree}
\AxiomC{$\Delta_1 \seqto \eta_1$}
\AxiomC{$\Delta_2,\Gamma,\eta_2 \seqto \exists \alpha \psi$}
\BinaryInfC{$\Delta_1,\Delta_2,\Gamma,\eta_1 \to \eta_2 \seqto \exists \alpha \psi$}
\end{prooftree}
By the induction hypothesis, if $\Delta$ and $\xi$ are true, then so are $\Delta_1,\Delta_2$ and $\eta_1 \to \eta_2$. Then, because $\Delta_1 \seqto \eta_1$ is valid we also have that $\eta_1$ is true. Combining this with the fact that $\eta_1 \to \eta_2$ is true shows that $\eta_2$ is true, which shows that the statement of the lemma also holds for the upper right sequent.

Finally, in case the topmost rule is a cut, the induction step follows from the fact that all cut formulas are true.
\end{proof}

\section{From realisability to Weihrauch reducibility}\label{sec-to-weihrauch}

In the previous section, we showed how to extract a realiser from a proof in $\mathrm{EL}_0$.
We now show how we can move from such realisers to Weihrauch reducibility.

\begin{thm}\label{thm-real-to-weihrauch}
If a $\Pi^1_2$-sequent $\zeta_1,\dots,\zeta_n \seqto \zeta_0$ has a strong realiser, then there are natural numbers $e_1,\dots,e_{n+1}$ such that $\zeta_0$ Weihrauch-reduces to the composition of $\zeta_n,\dots,\zeta_1$, as witnessed by $e_1,\dots,e_{n+1}$. If $\mathrm{RCA}_0$ proves the existence of this realiser, then it also proves that $\zeta_0$ Weihrauch-reduces to the composition of $\zeta_1,\dots,\zeta_n$.
\end{thm}
\begin{proof}
Let $\zeta_i = \forall \alpha_{2i+1} \xi_i(\alpha_{2i+1}) \exists \alpha_{2i+2} \psi_i(\alpha_{2i+1},\alpha_{2i+2})$, and let $(v,w)$ be a strong realiser.
For ease of notation, let us assume that $|v|=2n$, that $v_1(1) < v_1(3) < v_1(4) < \dots < v_1(2n+1) < v_1(2n+2) < v_1(2)$, and that $v(3),v(5),\dots,v(2n+1),v(2) \in \omega \times \omega$; the general case follows in the same way.
We reason within $\mathrm{RCA}_0$. Let $f_1$ be arbitrary such that $\xi_0(\alpha_{2i+1})$ holds. Next, let $f_3 = \Phi_{v_2(3)}(f_1)$ (which is total because $(v,w)$ is a strong realiser), and 
let $f_4$ be arbitrary such that $\psi_1(f_3,f_4)$ holds. Continuing like this, for $1 < i < n$ we let 
$f_{2i+1} = \Phi_{v_2(2i+1)}(f_1,f_3,f_4,\dots,f_{2i})$ and we let $f_{2i+2}$ be arbitrary such that $\psi_i(f_{2i+1},f_{2i+2})$ holds. Then $\xi_i(f_{2i+1})$ holds for every $1 \leq i \leq n$ because $(v,w)$ is a strong realiser.

Now, let $f_2 = \Phi_{v_2(2)}(f_1,f_3,f_4,\dots,f_{2n+2})$, which is again total because $(v,w)$ is a strong realiser. We claim: $(f_1,\dots,f_{2n+2})$ can be extended to a sequence which is valid for $(v,w)$ and which negatively realises $\zeta_1,\dots,\zeta_n$ and $\xi_0$.

For ease of notation, let us assume $w_1(1) < w_1(2) < \dots < w_1(|w|)$; otherwise we rename variables. Let us assume $a_1,\dots,a_{i-1}$ have been defined, and let us define $a_i$. If $w(i) = (k,e)$, let $a_i = \Phi_e(f_1,\dots,f_{2n+2},a_1,\dots,a_{i-1})$ if this is defined, and $0$ otherwise. If $w(i) \in \omega$, let $\eta$ be the unique subformula of $\zeta_0,\zeta_1,\dots,\zeta_n$ such that $\eta$ is of the form $Q x_i \eta'(x_i)$. If $\eta$ is a subformula of $\xi_j$ for some $j \geq 0$ or of $\psi_j$ for some $j \geq 1$, call this formula of which $\eta$ is a subformula $\xi$. Now remove from $\xi$ all quantifiers of the form $Q x_s$ for some $s < i$ to obtain $\xi'$. Next, choose a value $a_i$ such that $\xi'$ is true at $f_1,\dots,f_{2n+2},a_1,\dots,a_i$. We can always do this because $\xi_0,\xi_1,\dots,\xi_n$ and $\psi_1,\dots,\psi_n$ are true, and because how we have inductively chosen $a_1,\dots,a_{i-1}$.
On the other hand, if $\eta$ is a subformula of $\psi_0$, let $a_i$ be $0$.

It is now not hard to check that $(f_1,\dots,f_{2n+2},a_1,\dots,a_{|w|})$ is a valid sequence negatively realising $\zeta_1,\dots,\zeta_n$, and hence it positively realises $\zeta_0$. Since it also negatively realises $\xi_0$, we see that it positively realises $\psi_0(f_1,f_2)$. In particular, $\psi_0(f_1,f_2)$ is true, which completes the proof that $\zeta_0$ Weihrauch-reduces to the composition of $\zeta_1,\dots,\zeta_n$.
\end{proof}

We can now prove the first implication of our main theorem, by combining the theorems proven above.

\begin{thm}\label{thm-main-1}
If $\zeta_1 \to \zeta_0$ is a number-negative $\Pi^1_2$-sequent provable in $\mathrm{EL}_0 + \mathrm{MP}$, then there are natural numbers $n$ and $e_1,\dots,e_{n+1}$ such that $\mathrm{RCA}_0$ proves that $\zeta_0$ Weichrauch-reduces to the composition of $n$ copies of $\zeta_1$, as witnessed by $e_1,\dots,e_{n+1}$.
\end{thm}
\begin{proof}
Let $\Sigma$ be the set of formulas which are contraction-similar to an axiom of $\mathrm{EL}_0 + \mathrm{MP}$.
From Lemma \ref{lem-con-sim} we know that there is some sequent $\eta_1,\dots,\eta_n \seqto \psi$ which is contraction-similar to $\zeta_1 \seqto \zeta_0$ and which is provable in $(\mathrm{EL}_0 + \mathrm{MP})^\mathrm{a}$ from the axioms in $\Sigma$. Then, by Theorem \ref{thm-strong-realiser} there is a strong realiser of $\Gamma \seqto \psi$, where $\Gamma$ is a permutation of a subsequence of $\eta_n,\dots,\eta_1$. So, by Theorem \ref{thm-real-to-weihrauch} $\zeta_0$ Weihrauch-reduces to the composition of $\Gamma$. All these steps are provable in $\mathrm{RCA}_0$.

However, since every $\eta_i$ is contraction-similar to $\zeta_1$, we know from Lemma \ref{lem-con-sim-equiv} that $\zeta_1 \leftrightarrow \eta_i$; similarly we have that $\zeta_0 \leftrightarrow \psi$. Thus, we have that 
$\zeta_0$ is Weihrauch-reducible to the composition of finitely many copies of $\zeta_1$, as desired.
\end{proof}

\section{From Weihrauch reducibility to $\mathrm{EL}_0$}\label{sec-converse}

We now prove the reverse implication of our main theorem. For this, we use the Kuroda negative translation.

\begin{defi}{\rm (Kuroda \cite{kuroda-1951})}
The \emph{Kuroda negative translation} of $\phi$, $\phi^q$, is defined as $\neg\neg \phi^*$, where $\phi^*$ is recursively defined by:
\begin{itemize}
\item $A^* = A$ for atomic formulas A.
\item $(\phi \square \psi)^* = \phi^* \square \psi^*$ for $\square \in \{\wedge,\to\}$.
\item $(\exists x \phi)^* = \exists x \phi^*$ and $(\exists \alpha \phi)^* = \exists \alpha \phi^*$.
\item $(\forall x \phi)^* = \forall x \neg\neg \phi^*$ and $(\forall \alpha \phi)^* = \forall \alpha \neg\neg  \phi^*$.
\end{itemize}
\end{defi}

\begin{lem}
If $\mathrm{RCA}_0 \vdash \phi$, then $\mathrm{EL}_0 + \mathrm{MP} \vdash \phi^q$.
\end{lem}
\begin{proof}
See Fujiwara \cite[Lemma 11]{fujiwara-2015} or Kohlenbach \cite[Propositions 10.3 and 10.6]{kohlenbach-2005}; although they prove this result for $\mathrm{RCA}$ and $\mathrm{EL} + \mathrm{MP}$, it is easy to see that the same proof yields the desired result.
\end{proof}

\begin{thm}\label{thm-main-2}
Let $\zeta_i = \forall \alpha \xi_i(\alpha_i) \to \exists \beta_i \psi_i(\alpha_i,\beta_i)$ for $0 \leq i \leq n$ be $\Pi^1_2$-formulas.
If there are natural numbers $e_1,\dots,e_{n+1}$ such that $\mathrm{RCA}_0$ proves that $\zeta_0$ Weihrauch-reduces to the composition of $\zeta_n,\dots,\zeta_1$ with reductions witnessed by $e_1,\dots,e_{n+1}$, and we let $\zeta'_i$ be the formula $\zeta_i$ with $\xi_i$ replaced by $\xi_i^q$ and $\psi_i$ replaced by $\psi_i^q$,
then $\mathrm{EL}_0 + \mathrm{MP}$ proves that $(\zeta'_1 \wedge \dots \wedge \zeta'_n) \to \zeta'_0$.
\end{thm}
\begin{proof}
The fact that $\zeta_0$ Weihrauch-reduces to the composition of $\zeta_1,\dots,\zeta_n$, as witnessed by $e_1,\dots,e_n+1$, is expressed by a formula of the form
\[\forall \alpha_0,\dots,\alpha_n,\beta_0,\dots,\beta_n (\gamma(\alpha_0,\dots,\alpha_n,\beta_0,\dots,\beta_n)),\]
where $\gamma$ is built using only the propositional connectives from the formulas $\xi_i$ and $\psi_i$, from formulas expressing that some functional is defined on given inputs, and formulas expressing that some $\alpha_i$ or $\beta_i$ is equal to the output of some functional. Note that the latter two are $\Pi^0_2$ respectively $\Pi^0_1$. Also, note that for a $\Pi^0_2$-formula $\forall x \exists y \zeta$, we have (in $\mathrm{EL}_0 + \mathrm{MP}$) that
\[(\forall x \exists y \zeta)^q = \neg\neg \forall x \neg\neg \exists y \zeta \Leftrightarrow \forall x \neg\neg \exists y \zeta \Leftrightarrow \forall x \exists y \zeta,\]
where the last equivalence follows by Markov's principle. Thus, we see (in $\mathrm{EL}_0 + \mathrm{MP}$) that
\begin{align*}
(\forall \alpha_0,\dots,\alpha_n,\beta_0,\dots,\beta_n (\gamma))^q
&\Leftrightarrow \neg\neg \forall \alpha_0,\dots,\alpha_n,\beta_0,\dots,\beta_n \neg\neg \gamma^*\\
&\Leftrightarrow \forall \alpha_0,\dots,\alpha_n,\beta_0,\dots,\beta_n \neg\neg \gamma^*\\
&\Leftrightarrow \forall \alpha_0,\dots,\alpha_n,\beta_0,\dots,\beta_n \gamma^q\\
&\Leftrightarrow \forall \alpha_0,\dots,\alpha_n,\beta_0,\dots,\beta_n \gamma',
\end{align*}
where $\gamma'$ is the formula $\gamma$ where every formula $\xi_i$ is replaced by $\xi_i^q$ and every formula $\psi_i$ is replaced by $\psi_i^q$. Thus, using the previous lemma we see that $\mathrm{EL}_0 + \mathrm{MP}$ proves that $\zeta_0'$ Weihrauch-reduces to $\zeta'_1,\dots,\zeta'_n$. The result now easily follows.
\end{proof}

In particular, we can now prove our main result.

\begin{thm}\label{thm-main}
Let $\zeta_i = \forall \alpha_i \xi_i(\alpha_i) \to \exists \beta_i \psi_i(\alpha_i,\beta_i)$ for $i \in \{0,1\}$. Then the following are equivalent:
\begin{enumerate}
\item There are an $n \in \omega$ and $e_1,\dots,e_{n+1}$ such that $\mathrm{RCA}_0$ proves that $e_1,\dots,e_{n+1}$ witness that $\zeta_0$ Weihrauch-reduces to the composition of $n$ copies of $\zeta_1$.
\item For $\zeta'_i = \forall \alpha_i \xi^q_i(\alpha_i) \to \exists \beta_i \psi^q_i(\alpha_i,\beta_i)$ we have that $\mathrm{EL}_0 + \mathrm{MP}$ proves that $\zeta'_1 \to \zeta'_0$.
\end{enumerate}
\end{thm}
\begin{proof}
From Theorem \ref{thm-main-1} and Theorem \ref{thm-main-2}.
\end{proof}

\section{True Weihrauch reducibility and affine logic}

It turns out that, if we weaken our proof system for intuitionistic logic, we get a correspondence with \emph{true} Weihrauch reducibility, i.e.\ without resorting to the composition of multiple copies of $\zeta_1$. This gives us a proof-theoretic formulation of Weihrauch-reducibility in $\mathrm{RCA}_0$.

\begin{thm}\label{thm-main-affine}
Let $\zeta_i = \forall \alpha_i \xi_i(\alpha_i) \to \exists \beta_i \psi_i(\alpha_i,\beta_i)$ for $i \in \{0,1\}$. Then the following are equivalent:
\begin{enumerate}
\item\label{equiv1} There are $e_1,e_2$ such that $\mathrm{RCA}_0$ proves that $e_1,e_2$ witness that $\zeta_0$ Weihrauch-reduces to $\zeta_1$.
\item\label{equiv2} For $\zeta'_i = \forall \alpha_i \xi^q_i(\alpha_i) \to \exists \beta_i \psi^q_i(\alpha_i,\beta_i)$ we have that $(\mathrm{EL}_0 + \mathrm{MP})^{\exists\alpha\mathrm{a}}$ proves that $\zeta'_1 \to \zeta'_0$.
\end{enumerate}
\end{thm}
\begin{proof}
First, the fact that \eqref{equiv2} implies \eqref{equiv1} follows directly from the proof of Theorem \ref{thm-main-1}. Conversely, if \eqref{equiv1} holds, then, as in the proof of Theorem \ref{thm-main-2}, we know that $\mathrm{EL}_0 + \mathrm{MP}$ proves that $\zeta_0$ Weihrauch-reduces to $\zeta_1$. Since we can formulate this as the universal closure of an arithmetical formula $\eta$, we know that there is a free-cut free proof of $\eta$, and since $\eta$ does not contain function quantifiers this proof is in fact in $(\mathrm{EL}_0 + \mathrm{MP})^{\exists\alpha\mathrm{a}}$.\footnote{There might be a weakening with as principal formula a formula of the form $\exists \alpha \forall z B(x,\alpha(z))$, which is later a subformula of a principal formula of a cut, but it is not hard to see that such a pair of a weakening and a cut can be eliminated.}
Now it is not hard to verify that this proof can be extended to a proof of $\zeta_0 \to \zeta_1$ in $(\mathrm{EL}_0 + \mathrm{MP})^{\exists\alpha\mathrm{a}}$.
\end{proof}

\bibliographystyle{amsplain}

\newcommand{\noopsort}[1]{}
\providecommand{\bysame}{\leavevmode\hbox to3em{\hrulefill}\thinspace}
\providecommand{\MR}{\relax\ifhmode\unskip\space\fi MR }
\providecommand{\MRhref}[2]{%
  \href{http://www.ams.org/mathscinet-getitem?mr=#1}{#2}
}
\providecommand{\href}[2]{#2}

\end{document}